\documentclass[11pt]{amsart}
\usepackage{amssymb}
\usepackage{amsmath, amscd}
\usepackage{amsthm}
\usepackage[table,xcdraw]{xcolor}
\usepackage{ulem}
\usepackage{tikz}
\usepackage{tikz-cd}
\usepackage{circuitikz}
\usepackage[all]{xy}
\usepackage[colorlinks=true, linkcolor=blue,urlcolor=blue]{hyperref}

\newtheorem{theorem}{Theorem}[section]

\newtheorem{prop}[theorem]{Proposition}
\newtheorem{lemma}[theorem]{Lemma}

\newtheorem{corollary}[theorem]{Corollary}
\theoremstyle{definition}

\newtheorem{exam}[theorem]{Example}
\newtheorem{definition}[theorem]{Definition}

\newtheorem{problem}[theorem]{Problem}
\newtheorem{remark}[theorem]{Remark}

\begin{document}

\title{On Some Extensions of $\pi$-Regular Rings}

\author[P. Danchev]{Peter Danchev}
\address{Institute of Mathematics and Informatics, Bulgarian Academy of Sciences, 1113 Sofia, Bulgaria}
\email{danchev@math.bas.bg; pvdanchev@yahoo.com}
\author[A. Javan]{Arash Javan}
\address{Department of Mathematics, Tarbiat Modares University, 14115-111 Tehran Jalal AleAhmad Nasr, Iran}
\email{a.darajavan@modares.ac.ir; a.darajavan@gmail.com}
\author[A. Moussavi]{Ahmad Moussavi}
\address{Department of Mathematics, Tarbiat Modares University, 14115-111 Tehran Jalal AleAhmad Nasr, Iran}
\email{moussavi.a@modares.ac.ir; mpussavi.a@gmail.com}

\begin{abstract}
Some variations of $\pi$-regular and nil clean rings were recently introduced in \cite{5,8,7}, respectively. In this paper, we examine the structure and relationships between these classes of rings. Specifically, we prove that $(m, n)$-regularly nil clean rings are left-right symmetric and also show that the inclusions ($D$-regularly nil clean) $\subseteq$ (regularly nil clean) $\subseteq$ ($(m,n)$-regularly nil clean) hold, as well as we answer Questions 1, 2 and 3 posed in \cite{8}. Moreover, some other analogous questions concerning the symmetric properties of certain classes of rings are treated as well by proving that centrally Utumi rings are always strongly $\pi$-regular.
\end{abstract}

\maketitle

\section{Introduction and Preliminaries}

Throughout the present paper, all rings are assumed to be associative possessing an identity element. Our standard terminology and notations are mainly in agreement with those from \cite{17,18}. For example, the letters $U(R)$, $J(R)$, $Id(R)$, ${\rm Nil}(R)$ and $C(R)$ are stand to denote the set of units, the Jacobson radical, the set of idempotents, the set of nilpotents and the center of $R$, respectively. We also use the symbols $\mathbb{Z}$, $\mathbb{Z}_n$ and $\mathbb{N}$ to mean the set of integers, integers modulo $n$ and positive integers, respectively.

As usual, an element $r$ in a ring $R$ is called {\it regular} if there exists $x \in R$ such that $rxr = r$ or, equivalently, if the principal left ideal $Rr$ is a direct summand of $R$. The element $r$ is called {\it unit regular} if $rxr = r$ with $x \in U(R)$ or, equivalently, if $r = xe$ with $x \in U(R)$ and $e \in Id(R)$. Thus, a ring is called {\it (unit) regular} if every element of the ring is {\it (unit) regular}. If, however, $r = r^2x$ with $rx = xr$, these rings are known in the existing literature as {\it strongly regular} rings. Good sources for reviewing these three classes of rings are \cite{11} and \cite{23} as well as \cite{4}, \cite{13} and \cite{16}, respectively.

In that regard, as proper generalizations of the aforementioned concepts, it is worth mentioning that a ring $R$ is referred to as {\it $\pi$-regular} if, for each $r \in R$, there exists a positive integer $n$ such that $r^n \in r^n Rr^n$. Additionally, a ring $R$ is said to be {\it strongly $\pi$-regular} if, for every $r \in R$, there exists a positive integer $n$ such that $r^n \in r^{n+1} R \cap Rr^{n+1}$. The notion of strongly $\pi$-regular rings was initially introduced by Kaplansky in \cite{14}, and later these rings were studied by Azumaya in \cite{1} who referred to them as {\it right $\pi$-regular} rings. Moreover, Dischinger established in \cite{10} that the left-right symmetry of this notion is true. Consequently, every strongly $\pi$-regular ring is also $\pi$-regular, and the two concepts coincide in the case of commutative rings or, more generally, in the case of abelian rings (i.e., when all idempotents in the ring are central, that is, they lie in the center of the former ring). It is worth noticing that $\pi$-regularity has been extended in four different and non-trivial ways in the works of \cite{6,5,8,7}.

In the other vein, an element $r$ in a ring $R$ is known to be {\it clean} (see \cite{19}) if it can be expressed as the sum of an idempotent and a unit, that is, $r = e + u$, and so $R$ is called {\it clean} if every element in it is clean. If, in addition, $ue = eu$, then $R$ is called {\it strongly clean} (see \cite{20}).

Likewise, we remember that a ring $R$ is an exchange ring if, for each $a \in R$, there exists an idempotent $e \in Ra$ such that $1 - e \in R(1 - a)$ (see, for instance, \cite{19}). It was proved by Nicholson in
\cite[Theorem 2.1]{19} that a ring is an exchange ring if, and only if, all idempotents lift modulo every left ideal of the ring, and he also established the left-right symmetry of this concept. Furthermore, he also showed that every clean ring is an exchange ring. For a more concrete information regarding exchange and clean rings, we recommend the sources \cite{15} and \cite{20}, respectively.

On the same hand, an element $r$ in a ring $R$ is called {\it nil clean} if it can be expressed as the sum of an idempotent and a nilpotent, that is, $r = e + q$, and so $R$ is called {\it nil-clean} if every element in it is nil clean. If, in addition, $qe = eq$, then $R$ is called {\it strongly nil-clean} (see \cite{9}). It is straightforward to verify that nil clean rings are always clean, although the converse implication does not hold in general. This notion was further extended in \cite{6} to the concept of weakly nil clean rings that are those rings $R$ such that, for each $a \in R$, there exist $e \in Id(R)$ and $q \in {\rm Nil}(R)$ with $a - q - e \in eRa$.

On the other side, an element $r$ in a ring $R$ is called {\it Utumi} if $r-r^2x \in {\rm Nil}(R)$ for some $x \in R$, and thus $R$ is called {\it Utumi} if every element is Utumi. If, in addition, $rx = xr$, the ring $R$ is said to be {\it strongly Utumi}, whereas if $r - r^2x \in {\rm Nil}(R)\cap C(R)$, the ring $R$ is said to be {\it centrally Utumi}. It was established in \cite[Theorem 2.13]{5} that Utumi rings are necessarily left-right symmetric.

\medskip

We now need the following additional notions.

\begin{definition} \label{def 1}
Let R be an arbitrary ring.
\begin{enumerate}

\item

\cite[Definition 1.1]{5} -- A ring $R$ is called {\it regularly nil clean} if, for every $a \in R$, there exists $e\in Id(R) \bigcap Ra$ such that $a(1-e) = a-ae \in {\rm Nil}(R)$ (and hence that $(1-e)a \in {\rm Nil}(R)$).

\item

\medskip

\cite[Definition 1.1]{8} -- A ring $R$ is called {\it double regularly nil clean} or just {\it D-regularly nil clean} for short if, for each $a \in R$, there exists $e\in Id(R) \bigcap aRa$ such that $a(1-e) = a-ae \in {\rm Nil}(R)$ (and hence that $(1-e)a \in {\rm Nil}(R)$).

\item

\medskip

\cite[Definition 1.1]{7} -- A ring $R$ is called {\it $(m,n)$-regularly nil clean} for two non-negative integers $m,n$ if, for any $a \in R$, there exists $e\in Id(R) \bigcap a^mRa^n$ such that $a^m(1-e)a^n \in {\rm Nil}(R)$ (and hence that $a^n(1-e)a^m \in {\rm Nil}(R)$).

\end{enumerate}
\end{definition}

In this accent, in \cite[Problem 2.10]{7}, the author stated a question regarding the relationship between D-regularly nil clean rings and $(m, n)$-regularly nil clean rings. It was asked there of whether these concepts are independent of each other or {\it not}.

In what follows, we will demonstrate that, for every $m,n \in \mathbb{N}$, the implications are fulfilled:

\medskip

\begin{center}
$D$-regularly nil clean $\Longrightarrow$ regularly nil clean $\Longrightarrow$ $(m,n)$-regularly nil clean.
\end{center}

\medskip

However, when the ring is either abelian or NI, where in the latter situation ${\rm Nil}(R)$ forms an ideal of $R$, for any $m,n \in \mathbb{N}$ we have that:

\medskip

\begin{center}
D-regularly nil clean $\Longleftrightarrow$ regularly nil clean $\Longleftrightarrow$ $(m,n)$-regularly nil clean.
\end{center}

\medskip

The systematization of our further work is as follows: In the next section, we deal with the transversal between the so-called regularly nil clean and $(m,n)$-regularly nil clean rings for some non-negative integers $m$ and $n$. The main results here are Theorems~\ref{2-primal} and \ref{utumi}, respectively. In the subsequent third section, we are concerned with the connection between the so-termed regularly nil clean and D-regularly nil clean rings as here we succeeded to establish two major results that are, respectively, Theorems~\ref{Th3.3} and \ref{Th3.4}.

\section{The relationship between regularly nil clean and $(m,n)$-regularly nil clean rings}

We begin our work with the following technicality.

\begin{prop}
Suppose $m,n\in \mathbb{N}$. Then, $(m,n)$-regularly nil clean rings $R$ are left-right symmetric in the sense that, for any $a \in R$, there exists $f \in Id(R) \cap a^nRa^m$ with $a^n(1 - f)a^m \in {\rm Nil}(R)$.
\end{prop}

\begin{proof}
In view of Definition \ref{def 1} above, given that $e \in a^mRa^n$ such that $a^n(1-e)a^m \in {\rm Nil}(R)$. Without loss of generality, we may assume that $m<n$. Write $e=a^mra^n$, and set $f :=(a^nra^m)^2$. Thus,

\begin{align*}
  f^2 & = (a^nra^m)(a^nra^m)(a^nra^m)(a^nra^m)  \\
      & = a^{n-m}(a^mra^n)(a^mra^n)(a^mra^n)(a^mra^m) \\
      & = a^{n-m}(a^mra^n)(a^mra^m)=(a^nra^m)(a^nra^m)=f.
\end{align*}

\noindent Also,

$$fa^n=(a^nra^m)(a^nra^m)a^n=a^{n-m}(a^mra^n)(a^mra^n)a^m=a^{n-m}ea^m.$$ This allows us to infer that
$$a^m(1-f)a^n= a^m(a^n-fa^n)= a^m(a^n-a^{n-m}ea^m)=$$

$$=a^na^m-a^nea^m=a^n(1-e)a^m \in {\rm Nil}(R).$$ Therefore, the relation $a^n(1-f)a^m \in {\rm Nil}(R)$ follows now immediately.
\end{proof}

Our next statement we are planning to prove is the following.

\begin{prop} \label{mn regular}
Regularly nil clean rings are $(m,n)$-regularly nil clean for every $m,n \in \mathbb{N}$.
\end{prop}

\begin{proof}
Suppose that $R$ is a regularly nil clean ring. Let $a \in R$ be an arbitrary element and set $k:=n+m$, so $a^k \in R$. By definition, there is $e\in a^kR$ such that $(1-e)a^k \in {\rm Nil}(R)$. Write $e= a^kr$ for some $r \in R$, and put $r':=re$ and $f:=a^mr'a^n$. Thus, $e= a^kr'$, $r'=r'e$ and $$f^2= a^mr'a^kr'a^n= a^mr'ea^n= a^mr'a^n= f.$$ Also, $$a^nf=a^na^mr'a^n=a^kr'a^n=ea^n.$$ Then, we have $$a^n(1-f)a^m= (a^n-a^nf)a^m= (a^n-ea^n)a^m= (1-e)a^k \in {\rm Nil}(R).$$ Therefore, $a^m(1 - f)a^n \in {\rm Nil}(R)$, and so the result follows.
\end{proof}

We are now prepared to establish the following assertion.

\begin{prop}
Let $R$ be a ring and $m,n,p,q \in \mathbb{N} \cup \{0\}$ with $m+n = p+q$. Then, $R$ is an $(m, n)$-regularly nil clean ring if, and only if, $R$ is a $(p,q)$-regularly nil clean.
\end{prop}

\begin{proof}
Suppose that $e \in a^mRa^n$ such that $(1-e)a^{m+n} \in {\rm Nil}(R)$. Without loss of generality, we assume that $m>p$. Write $e=a^mra^n$, and define $f := a^pra^{p+q}ra^q \in a^pRa^q$. Then,

\begin{align*}
  f^2 & = (a^pra^{p+q}ra^q)(a^pra^{p+q}ra^q) = (a^pra^{m+n}ra^q)(a^pra^{m+n}ra^q)  \\
      & = a^pra^n(a^mra^n)(a^mra^n)(a^mra^n)a^{q-n} = a^pra^n(a^mra^n)a^{q-n}\\
      & = a^pra^{n+m}ra^q = a^pra^{p+q}ra^q = f.
\end{align*}

\noindent However, we can see that $a^{m-p}f = ea^{q-n}$. Also, $$a^{m - p}(1 - f)a^{n+p} = (a^{m-p} - a^{m-p}f)a^{n+p} = (a^{m-p} - ea^{q-n})a^{n+p} =$$

$$=a^{m+n} - ea^{p+q} =  a^{m+n} - ea^{m+n} = (1-e)a^{m+n} \in {\rm Nil}(R).$$ So, $$a^{m - p}(1 - f)a^{n+p} \in {\rm Nil}(R),$$ but $$a^p(1-f)a^q \in {\rm Nil}(R)$$ if, and only if, $$a^{p+q}(1-f) \in {\rm Nil}(R)$$ if, and only if, $$a^{m+n}(1-f) \in {\rm Nil}(R)$$ if, and only if, $$a^{m+n+p-p}(1-f) \in {\rm Nil}(R)$$ if, and only if, $$a^{m-p}(1-f)a^{n+p} \in {\rm Nil}(R).$$ Consequently, $R$ is a $(p, q)$-regularly nil clean ring, as asserted.
\end{proof}

As a consequence, we derive:

\begin{corollary}
Let $R$ be a ring and $m,n \in \mathbb{N} \cup \{0\}$. The following conditions are equivalent:
 \begin{enumerate}

\item[\textsc{1.}] $R$ is an $(m, n)$-regularly nil clean ring.

\item[\textsc{2.}] For every $a \in R$, there exists $e=e^2 \in a^{n+m}R$ such that $a^{n+m}(1-e) \in {\rm Nil}(R)$.

\item[\textsc{3.}] For every $a \in R$, there exists $e=e^2 \in Ra^{n+m}$ such that $a^{n+m}(1-e) \in {\rm Nil}(R)$.

\item[\textsc{4.}] For every $a \in R$, $a^{n+m}$ is a regularly nil clean element.
\end{enumerate}
\end{corollary}

The following technical claim is useful. Recall that a ring $R$ is said to be {\it reduced}, provided ${\rm Nil}(R)=\{0\}$, and is said to be {\it abelian}, provided ${\rm Id}(R)\subseteq C(R)$. Notice that reduced rings are always abelian and this implication is generally irreversible.

\begin{lemma} \label{mn rgular is regular}
Let $R$ be a reduced $(m,n)$-regularly nil clean ring such that $(m,n)\neq (0,0)$. Then, $R$ is a regular ring.
\end{lemma}

\begin{proof}
Assume $a \in R$ with $m\geq n$ (and hence $m>0$). Thus, there is $e \in a^nRa^m \subseteq Ra$ such that $a^{n+m}(1-e) \in \text{Nil}(R)=\{0\}$. Since $R$ is abelian, we deduce $(a(1-e))^{n+m}=0$, so $a(1-e)=0$. Therefore, $a=ara$ for some $r \in R$, as promised.
\end{proof}

We now can attack the following basic statement which focuses on some common extensions of commutative rings. In it, in defining a $2$-primal ring, one usually means a ring  $R$ such that the set of nilpotent elements ${\rm Nil}(R)$ equals the prime radical (that is, the intersection of prime ideals). In this direction, the symbol ${\rm Nil}_{*}(R)$ stands for the {\it lower} nil-radical of such a ring $R$.

\begin{theorem}\label{2-primal}
Let $R$ be a $2$-primal $(m,n)$-regularly nil clean ring for some $m,n \in \mathbb{N} \cup \{0\}$ such that $(m,n)\neq (0,0)$. Then, ${\rm M}_k(R)$ is a $(p,q)$-regularly nil clean ring for every $p,q \in \mathbb{N} \cup \{0\}$ and $k \in \mathbb{N}$.
\end{theorem}

\begin{proof}
We know that (see, for more details, \cite{17,18})
$${\rm M}_k(R)/J({\rm M}_k(R)) \cong {\rm M}_k(R/J(R))={\rm M}_k(R/\text{Nil}_{*}(R)).$$
Since the quotient $\overline{R}=R/\text{Nil}_{*}(R)$ is simultaneously reduced and $(m, n)$-regularly nil clean, Lemma \ref{mn rgular is regular} ensures that $\overline{R}$ is regular. Therefore, \cite{17} assures that ${\rm M}_k(\overline{R})$ is regular. Furthermore, combining \cite[Proposition 2.1]{5} and Proposition \ref{mn regular}, one sees that ${\rm M}_k(\overline{R})$ is a $(p, q)$-regularly nil clean ring.

On the other hand, since $J(R)$ is a nil-ideal and $R$ is $2$-primal, one verifies that $$J({\rm M}_k(R))={\rm M}_k(J(R))={\rm M}_k(\text{Nil}_{*}(R))=\text{Nil}_{*}({\rm M}_k(R))$$ is nil. Consequently, \cite[Lemma 2.1]{7} insures that ${\rm M}_k(R)$ is $(p, q)$-regularly nil clean, as claimed.
\end{proof}

The next two technical claims are helpful.

\begin{lemma}
Let $R$ be an abelian $(m,n)$-regularly nil clean ring for some $m,n\in \mathbb{N}\cup \{0\}$ such that $(m,n)\neq (0,0)$. Then, $R$ is $(p,q)$-regularly nil clean for every $p,q \in \mathbb{N}\cup \{0\}$.
\end{lemma}

\begin{proof}
If, for a moment, $m+n \le  p+q$, then there is $k \in \mathbb{N}$ such that $p+q \le k(m+n)$. For each $a \in R$, there is $e \in (a^k)^{n+m}R \subseteq a^{p+q}R$ such that $a^{k(n+m)}(1-e) \in \text{Nil}(R)$. Therefore, one checks that
\begin{align*}
&(a^{p+q}(1-e))^{k(n+m)}=((a(1-e))^{p+q})^{k(m+n)}=((a(1-e))^{k(m+n)})^{p+q}\\
&=(a^{k(n+m)}(1-e))^{p+q} \in \text{Nil}(R).
\end{align*}
Hence, $a^{p+q}(1-e) \in \text{Nil}(R)$, as needed.

If, however, $p+q \le  m+n$, then, for each $a \in R$, there is $e \in a^{n+m}R \subseteq a^{p+q}R$ such that $a^{n+m}(1-e) \in \text{Nil}(R)$. Therefore,
\begin{align*}
&(a^{p+q}(1-e))^{n+m}=((a(1-e))^{p+q})^{n+m}=((a(1-e))^{n+m})^{p+q}\\
&=(a^{n+m}(1-e))^{p+q} \in \text{Nil}(R).
\end{align*}
Hence, $a^{p+q}(1-e) \in \text{Nil}(R)$, and thus we are set.
\end{proof}

Recall once again for a readers' convenience that a ring $R$ is said to be {\it NI}, provided ${\rm Nil}(R)$ forms an ideal of $R$.

\medskip

The next technical claim considerably extends Lemma~\ref{mn rgular is regular}.

\begin{lemma}
Let $R$ be an NI ring which is $(m,n)$-regularly nil clean for some $m,n\in \mathbb{N}\cup \{0\}$ such that $(m,n)\neq (0,0)$. Then, $R$ is $(p,q)$-regularly nil clean for every $p,q \in \mathbb{N}\cup \{0\}$.
\end{lemma}

\begin{proof}
Firstly, if $m+n \le p+q$, then there is $k \in \mathbb{N}$ such that $p+q \le k(n+m)$. One inspects that, for each $a \in R$, there is $$e=e^2 \in  (a^k)^{n+m}R \subseteq a^{p+q}R, \text{ such that } (a^k)^{n+m}(1-e) \in {\rm Nil}(R).$$ Letting $x$ be a remainder of $k(n+m)$ modulo $p+q$ ($0\leq x<p+q$), there will exist $k' \in \mathbb{N}$ such that $$a^{k'(p+q)}a^x(1-e) \in {\rm Nil}(R).$$ Therefore, $$a^{k'(p+q)}(1-e)a^x \in {\rm Nil}(R).$$ But, as $R$ is an NI ring, we find that $$a^{k'(p+q)}(1-e)a^{p+q} \in {\rm Nil}(R).$$
Furthermore, using the same property of $R$, namely that ${\rm Nil}(R)$ is an ideal of $R$, we subsequently obtain that

\begin{align*}
    &a^{k'(p+q)}(1-e)a^{p+q} \in {\rm Nil}(R)  \Rightarrow a^{k'(p+q)}(1-e)a^{p+q} (1-e) \in {\rm Nil}(R) \\
    &a^{(k'-1)(p+q)}a^{p+q}(1-e)a^{p+q}(1-e)\in {\rm Nil}(R) \Rightarrow  a^{(k'-1)(p+q)}[a^{p+q}(1-e)]^2 \in {\rm Nil}(R) \\
    &\Rightarrow \cdots \Rightarrow  [a^{p+q}(1-e)]^{k'+1} \in {\rm Nil}(R) \Rightarrow  a^{p+q}(1-e) \in {\rm Nil}(R).
\end{align*}

In fact, a critical step is that $$a^{(k'-1)(p+q)}[a^{p+q}(1-e)]^2 \in {\rm Nil}(R) = a^{p+q}a^{(k'-2)(p+q)}[a^{p+q}(1-e)]^2 \in {\rm Nil}(R)$$ yields $a^{(k'-2)(p+q)}[a^{p+q}(1-e)]^2a^{p+q} \in {\rm Nil}(R)$ and thus, as above illustrated, multiplying both sides from the right by $1-e$ and proceeding by decreasing the value of $k'$, we can get the final result, as pursued.

Hence, one concludes also that $(1-e)a^{p+q} \in {\rm Nil}(R)$, as expected.

\medskip

Reciprocally, if $p+q \le n+m$, similarly to the demonstrated above tricks, one infers that $a^{p+q}(1-e) \in {\rm Nil}(R)$, as asked for.
\end{proof}

\begin{remark} In view of Proposition~\ref{mn regular}, it is conceptually enough to prove in the previous two lemmas only that $R$ is a regularly nil clean ring. We, however, gave more direct proofs.
\end{remark}

As three consequences, we obtain:

\begin{corollary}
Let $R$ be a 2-primal $(m,n)$-regularly nil clean ring for some $m,n\in \mathbb{N}\cup \{0\}$ such that $(m,n)\neq (0,0)$. Then, for every $k \in \mathbb{N}$, ${\rm M}_k(R)$ is a regularly nil clean ring.
\end{corollary}

\begin{corollary}
Let $R$ be an abelian $(m,n)$-regularly nil clean ring for some $m,n\in \mathbb{N}\cup \{0\}$ such that $(m,n)\neq (0,0)$. Then, $R$ is regularly nil clean.
\end{corollary}

\begin{corollary}
Let $R$ be an NI $(m,n)$-regularly nil clean ring for some $m,n\in \mathbb{N}\cup \{0\}$ such that $(m,n)\neq (0,0)$. Then, $R$ is regularly nil clean.
\end{corollary}

We now continue our work on another vein. In fact, it was asked in \cite{5} whether Utumi's rings possess the left-right symmetric property. It was shown in \cite[Theorem 2.13]{8} that the answer to this question is positive. Nevertheless, in the next technical claim, we shall show in a different and more transparent way that this is really true.

\begin{lemma}
For any ring R, the following are equivalent:

\begin{enumerate}

\item[\textsc{1.}] $R$ is an Utumi ring.

\item[\textsc{2.}] For every $x \in R$, there exists $y \in R$ such that $x-x^2y \in {\rm Nil}(R)$.

\item[\textsc{3.}] For every $x \in R$, there exists $y \in R$ such that $x-xyx \in {\rm Nil}(R)$.

\item[\textsc{4.}] For every $x \in R$, there exists $y \in R$ such that $x-yx^2 \in {\rm Nil}(R)$.

\end{enumerate}
\end{lemma}

\begin{proof}
$(1)\Leftrightarrow (2)$ is clear from the definition. It is now enough to show that the implications $(2)\Rightarrow(3)\Rightarrow(4)\Rightarrow(2)$ are valid. To that goal, for any two elements $x, y$ of a ring $R$, it follows that $xy \in {\rm Nil}(R)$ if, only if, $yx \in {\rm Nil}(R)$, whence

\begin{align*}
x-x^2y \in {\rm Nil}(R) & \Leftrightarrow  x(1-xy) \in {\rm Nil}(R)  \Leftrightarrow  (1-xy)x \in {\rm Nil}(R) \\
       & \Leftrightarrow \,\,  x-xyx \in {\rm Nil}(R) \,\,  \Leftrightarrow  x(1-yx) \in {\rm Nil}(R) \\
       & \Leftrightarrow  (1-yx)x \in {\rm Nil}(R)  \Leftrightarrow \:\: x-yx^2 \in {\rm Nil}(R).
\end{align*}
This automatically completes the proof.
\end{proof}

Besides, it was asked in \cite[Problem 3.3]{5} whether centrally Utumi's rings are exchange. We now will concentrate on this and will prove even something more, namely that each centrally Utumi's ring is strongly $\pi$-regular, and since it was demonstrated in \cite[Example 23]{22} that strongly $\pi$-regular rings are exchange, we will be done.

\begin{theorem}\label{utumi}
Any centrally Utimi ring is a strongly $\pi$-regular ring.
\end{theorem}

\begin{proof}
Let $R$ be a centrally Utimi ring and let $x \in R$ be an arbitrary element. Hence, there is $y \in R$ such that $x-x^2y \in {\rm Nil}(R)\cap C(R)$. Therefore, for each $n \in \mathbb{N}$, we derive:
\begin{align*}
(x - x^2y)^n &= (x - x^2y)^{n-1}(x - x^2y) = (x - x^2y)^{n-1}x(1 - xy)\\
&= x(x - x^2y)^{n-1}(1 - xy)= x(x - x^2y)^{n-2}(x - x^2y)(1 - xy) \\
&= x(x - x^2y)^{n-2}x(1 - xy)^2 = x^2(x - x^2y)^{n-2}(1 - xy)^2\\
&= \cdots = x^n(1 - xy)^n.
\end{align*}

On the other hand, there is $m \in \mathbb{N}$ such that $(x-x^2y)^{m} = 0$. So, $x^{m}(1-xy)^{m} = 0$ and expanding this by the classical Newton's binomial formula, we deduce that

\begin{align*}
0 & = x^m(1 - xy)^m = x^m \sum_{i=0}^{m}(-1)^{i}\binom{m}{i}(xy)^{i} \\
    & =  x^m (1+ \sum_{i=1}^{m}(-1)^{i}\binom{m}{i}x(yx)^{i-1}y ) =  x^m + \sum_{i=1}^{m}(-1)^{i}\binom{m}{i}x^{m+1}(yx)^{i-1}y,
\end{align*}
which gives that

$$x^{m} = k_1x^{m+1}y+\cdots + k_{m}x^{m+1}(yx)^{m-1}y,$$
where we put $k_i = (-1)^{i+1}\binom{m}{i}\in \mathbb{N}$ whenever $i = 1,\ldots , m$. Finally, one detects that $x^{m} = x^{m+1}r$, where $r \in R$. This shows that the ring $R$ is strongly $\pi$-regular, as promised.
\end{proof}

Meanwhile, in order to illustrate that the reverse implication is untrue, it is easy to see that the ring $\mathbb{Z}_{(2)}$ consisting of all rational numbers with odd denominators is a commutative (local) exchange ring which is, manifestly, not centrally Utimi.

\medskip

The next construction is worthy of documentation.

\begin{exam}
For any ring $R$, both the polynomial ring  $R[x]$ and the power series ring $R[[x]]$ are {\it not} $(m,n)$-regularly nil clean for every $m,n \in \mathbb{N}$.
\end{exam}

\begin{proof}
If we assume, by way of a contradiction, that the ring $R[[x]]$ is $(m,n)$-regularly nil clean, then according to \cite[Exercise 5.6]{18} we know that $$J(R[[x]]) = \{a + xf(x) | a \in J(R) \: \textrm{and} \: f(x) \in R[[x]]\}.$$ So, $x \in J(R[[x]])$. Therefore, $J(R[[x]])$ is not nil, but by \cite[Lemma 2.1]{7}, $J(R[[x]])$ is known to be nil. This is, indeed, a contradiction, as required.

Now, assume, in a way of a contradiction, that $R[x]$ is $(m,n)$-regularly nil clean. Then, there exists $e \in x^mR[x]x^n \cap Id(R[x])$ with $x^{n+m}(1-e) \in {\rm Nil}(R[x])$. Consequently, there is $f \in R[x]$ such that $e=fx^{m+n}$. With a simple calculation at hand, we verify that $e=0$. But, one checks that $$x^{n+m}(1-e)=x^{n+m} \notin {\rm Nil}(R[x]),$$ which is the desired contrary to our assumption.
\end{proof}

\section{The relationship between regularly nil clean and D-regularly nil clean rings}

According to the corresponding definitions alluded to above, every D-regularly nil clean is a regularly nil clean ring. In this section, we examine under what additional conditions and circumstances the opposite is true, that is, D-regularly nil clean rings and regularly nil clean rings do coincide.

\medskip

However, first and foremost, we provide some basic examples.

\begin{exam}\label{e3.1}  $\newline$
\begin{enumerate}

\item[\textsc{1.}] Each invertible element in a ring is D-regularly nil clean. In fact, if $a$ is invertible in $R$, we can write $e=1=aa^{-2}a \in aRa$. In particular, every division ring is D-regularly nil clean. \label{e3.1.1}

\item[\textsc{2.}] Each idempotent element in a ring is D-regularly nil clean. In fact, if $a$ is an idempotent in $R$, we can write $e=a=aa \in aRa$. In particular, any Boolean ring is D-regularly nil clean.

\item[\textsc{3.}] Each nilpotent element in a ring is D-regularly nil clean. In fact, if $a$ is a nilpotent in $R$ and $a^n=0$ for some $n\in \mathbb{N}$, we can write $e=0=aa^{n-2}a \in aRa$. \label{e3.1.3}

\item[\textsc{4.}] All D-regularly nil clean rings are closed under finite direct products. \label{e3.1.4}

\end{enumerate}
\end{exam}

With the first and third parts of the above example at hand, we can conclude that $\mathbb{Z}_{p^{n}}$ is a D-regularly nil clean ring for every prime number $p$ and positive integer $n$. Also, from the fourth part, we have that $\mathbb{Z}_{m}$ is a D-regularly nil clean ring for every positive integer $m$.

\medskip

Besides, it was asked in \cite[Problem 3.1]{8} whether an infinite direct product of (strongly) $\pi$-regular rings is a D-regularly nil clean ring. Owing to the next result, we answer this question in the negative.

\begin{exam}
An infinite product of (strongly) $\pi$-regular rings is {\it not} D-regularly nil clean. Even more, we can construct an infinite direct product of D-regularly nil clean rings that is {\it not} D-regularly nil clean itself. Indeed, consider $$R =\mathbb{Z}_2 \times \mathbb{Z}_4 \times \mathbb{Z}_8 \times \cdots .$$ We claim that the element $a=(0,2,2,2,\ldots) \in J(R)\subseteq R$ is {\it not} D-regularly nil clean as for otherwise it must be that $e \in aRa \subseteq J(R)$. So, it immediately follows that $e = 0$. But, this forces that $a(1-e)=a \notin {\rm Nil}(R)$, as asked for.
\end{exam}

We now have all the ingredients necessary to prove the following.

\begin{theorem}\label{Th3.3}
Let $R$ be an NI ring. Then, $R$ is regularly nil clean if, and only if, $R$ is D-regularly nil clean.
\end{theorem}

\begin{proof}
One elementarily sees that, it is enough to show only that every regularly nil clean ring is $D$-regularly nil clean. Invoking Proposition \ref{mn regular}, every regularly nil clean ring is $(1,1)$-regularly nil clean. Letting $a \in R$ be an arbitrary element, there exists $e=e^2 \in aRa$ such that $a(1-e)a \in {\rm Nil}(R)$. Since $R$ is an NI ring, we easily see that $a(1 - e) \in {\rm Nil}(R)$, as required.
\end{proof}

Having the needed machinery at hand, we are now in a position to prove the following.

\begin{theorem}\label{Th3.4}
Let $R$ be an abelian ring. Then, $R$ is regularly nil clean if, and only if, $R$ is D-regularly nil clean.
\end{theorem}

\begin{proof}
As above, it suffices to show only that every regularly nil clean ring is $D$-regularly nil clean. Exploiting Proposition \ref{mn regular}, each regularly nil clean ring is $(1,1)$-regularly nil clean. Therefore, for every $a \in R$, there exists $e=e^2 \in aRa$ such that $a(1-e)a \in {\rm Nil}(R)$. Since $R$ is an abelian ring, we have

\begin{align*}
a(1-e)a \in {\rm Nil}(R) &\Rightarrow  a(1-e)^2a \in {\rm Nil}(R) \Rightarrow  a(1-e)a(1-e) \in {\rm Nil}(R) \\
                       &\Rightarrow  (a(1-e))^2 \in {\rm Nil}(R) \Rightarrow  a(1-e) \in {\rm Nil}(R),
\end{align*}
as needed.
\end{proof}

We next proceed by showing the following.

\begin{lemma} \label{lem,chen} \cite[Lemma 1.5]{3}
Let $R$ be a commutative domain and $A \in {\rm M}_2(R)$. Then, $A$ is an idempotent if, and only if, either $A = 0$, or $A = I$, or
$A =
   \begin{pmatrix}
     a & b \\
     c & 1-a
   \end{pmatrix},
$
where $bc = a - a^2$ in $R$.
\end{lemma}

Taking into account \cite[Theorem 2.7]{8} and \cite[Theorem 2.10]{8}, regular and $\pi$-regular rings are both D-regularly nil clean. However, it was asked in \cite[Problem 3.3]{8} if there is a ($\pi$-)regular element that is {\it not} D-regularly nil clean.

\medskip

The following example manifestly illustrates that the answer to this question is absolutely positive (for results of a similar aspect, we refer to \cite{H} and \cite{K}, respectively).

\begin{exam}
There is a ($\pi$-)regular element which is {\it not} D-regularly nil clean.
\end{exam}

\begin{proof}
Let $R = {\rm M}_2(\mathbb{Z})$, and take

\begin{center}
$A=
  \begin{pmatrix}
    1 & -1 \\
    3 & -3
  \end{pmatrix} \in R
  $
  \quad \textrm{and} \quad
$B=
  \begin{pmatrix}
    1 & 0 \\
    0 & 0
  \end{pmatrix} \in R
  $
\end{center}

\noindent It is not so hard to inspect that $A=ABA$ and $A^n=(-2)^{n-1}A$ for all $n \in \mathbb{ N}$. If, however, we assume the contrary that $A$ is a D-regularly nil clean element, then there exists an element $E \in ARA \cap Id(R)$ such that $A(I-E) \in {\rm Nil}(R)$. Since $A$ is not nilpotent, we have $E \neq 0$ and write

$$E=\begin{pmatrix}
    1 & -1 \\
    3 & -3
  \end{pmatrix}
  \begin{pmatrix}
    x & y \\
    z & d
  \end{pmatrix}
  \begin{pmatrix}
    1 & -1 \\
    3 & -3
  \end{pmatrix}
  =
  \begin{pmatrix}
    h & -h \\
    3h & -3h
  \end{pmatrix},
$$
where $h=(x-z)+3(y-d) \in \mathbb{Z}$. In accordance with Lemma \ref{lem,chen}, we deduce that $E$ is not idempotent, which is an obvious contradiction with our assumption, thus giving the wanted claim.
\end{proof}

Bearing in mind the above example, one concludes also that a regularly nil clean element is {\it not} D-regularly nil clean.

\medskip

Resuming, the general relationship between regular rings, D-regularly nil clean rings, exchange rings and clean rings is as follows below:

\medskip

\begin{center}
\begin{tikzcd} \label{diagram2}
    &   &  \textrm{clean} \ar[d,Rightarrow] \\
  \textrm{regular} \ar[r,Rightarrow ] & \textrm{D-regularly nil clean} \ar[r,Rightarrow] & \textrm{exchange}

\end{tikzcd}
\end{center}

\medskip

These implications are known to be non-reversible in general. For example, the ring $\mathbb{Z}_4$ is D-regularly nil clean but plainly {\it not} a regular ring. Moreover, it is not too difficult to see that the ring $\mathbb{Z}_{(2)}$ is a (clean) exchange but {\it not} a D-regularly nil clean ring, because $J(R)$ is surely not nil (see \cite[Lemma 2.1]{8}).

\medskip

Likewise, it was asked in \cite[Problem 3.2]{8} whether there is a D-regularly nil clean ring which is {\it not} strongly clean (and even {\it not} clean). The next construction settles this in the affirmative.

\begin{exam}
There is a D-regularly nil clean ring that is neither clean {\it nor} strongly clean.
\end{exam}

\begin{proof} Knowing an example attributed to Bergman, namely \cite[Example 2]{12}, of a regular ring $R$ which is not generated by its units, we find by virtue of \cite[Theorem 2.7]{8} that such an $R$ is D-regularly nil clean. But, $R$ is definitely non-clean consulting with \cite[Proposition 10]{2}; thereby, is is not strongly clean too.

On the other side, in \cite{21} was constructed an example of a unit-regular ring which is {\it not} strongly clean. Therefore, thanks to \cite[Theorem 2.7]{8}, this ring is necessarily D-regularly nil clean which need {\it not} be strongly clean, as pursued.
\end{proof}

We are now ready to state and prove the following claim.

\begin{lemma} \label{Lemma 3.8}
Let $R$ be an NI ring and $a\in R$. If $a = b + q$, where $b$ is regularly nil clean and $q \in {\rm Nil}(R)$, then $a$ is regularly nil clean.
\end{lemma}

\begin{proof}
Choose an idempotent $g \in bR$ such that $(1-g)b \in {\rm Nil}(R)$. Write $g=br$ and set $\bar{g}:= br+{\rm Nil}(R)$. Since $a\equiv b$ modulo ${\rm Nil}(R)$, we obtain $\bar{g}= ar+{\rm Nil}(R)\in R/{\rm Nil}(R)$. Now, utilizing \cite[Theorem 21.28]{17}, there will exist an idempotent $e \in arR$ with $\bar{e} = \bar{g}$. But, one readily checks that $(1-e)a\equiv (1-g)b =0$ modulo ${\rm Nil}(R)$, as expected.
\end{proof}

We end up our work with the following difficult question.

\begin{problem} Give an explicit example showing that the classes of $(p,q)$-regularly nil clean rings and  $(m,n)$-regularly nil clean rings are different provided that $m+n\not= p+q$.
\end{problem}

\medskip
\medskip

\noindent{\bf Acknowledgement}. The authors are deeply thankful to the anonymous expert referee for their constructive suggestions made.

\medskip
\medskip

\section*{Funding}

The scientific work of the first-named author (P.V. Danchev) was supported in part by the Junta de Andaluc\'ia under Grant FQM 264.

\end{document}